\documentclass[reqno,a4paper,12pt]{amsart}
\usepackage[top=32mm, bottom=32mm, left=32mm, right=32mm]{geometry}
\usepackage{mathrsfs,enumerate}
\usepackage{amssymb}
\usepackage[colorlinks=true]{hyperref}

\newtheorem{theorem}{Theorem}

\theoremstyle{definition}

\newtheorem{question}[theorem]{Question}

\theoremstyle{remark}
\newtheorem{remark}[theorem]{Remark}

\begin{document}

\title[Invariance of distributional chaos for backward shifts]{Invariance of distributional chaos for backward shifts}


\author[X. Wu]{Xinxing Wu}
\address[X. Wu]{School of Sciences, Southwest Petroleum University, Chengdu, Sichuan 610500, P.R. China}
\email{wuxinxing5201314@163.com}

\author[Y. Luo]{Yang Luo}
\address[Y. Luo]{School of Sciences, Southwest Petroleum University, Chengdu, Sichuan, 610500, P.R. China}
\email{email of the Second Author}




\date{\today}



\begin{abstract}
A sufficient and necessary condition ensuring that the backward shift operator on the K\"{o}the sequence space admits an invariant distributionally
$\varepsilon$-scrambled set for some $\varepsilon>0$ is obtained, improving the main results in [F. Mart\'{\i}nez-Gim\'{e}nez, P. Oprocha, A. Peris,
J. Math. Anal. Appl., {\bf 351} (2009), 607--615].
\end{abstract}

\maketitle

Let $\mathbb{N}=\{1, 2, 3, \ldots\}$ and $\mathbb{Z}^+=\{0, 1, 2, \ldots\}$. According to \cite{Kothe}, an infinite matrix
$A=(a_{j,k})_{j, k\in \mathbb{N}}$ is called a {\it K\"{o}the matrix} if for every $j\in \mathbb{N}$ there exists some
$k\in \mathbb{N}$ with $a_{j, k}>0$ and $0\leq a_{j, k}\leq a_{j, k+1}$ for all $j, k\in \mathbb{N}$.

Consider the {\it backward shift} defined by
$$
B(x_{1}, x_{2}, x_{3}, \ldots)=(x_{2}, x_{3}, x_{4}, \ldots)
$$
on the K\"{o}the sequence space $\lambda_{p}(A)$ determined by a K\"{o}the matrix $A$, where, for $1\leq p<+\infty$,
$$
\lambda_{p}(A):=\left\{x\in \mathbb{K}^{\mathbb{N}}:
\|x\|_{k}:=\left(\sum_{j=1}^{\infty}|x_{j}a_{j,
k}|^{p}\right)^{1/p}<\infty, \forall k\in \mathbb{N}\right\},
$$
and, for $p=0$,
$$
\lambda_{0}(A):=\left\{x\in \mathbb{K}^{\mathbb{N}}:
\lim_{j\rightarrow \infty}x_{j}a_{j, k}=0, \|x\|_{k}:=\sup_{j\in
\mathbb{N}}|x_{j}a_{j, k}|, \forall k\in \mathbb{N}\right\}.
$$
It is possible to define a complete metric on $\lambda_{p}(A)$ which is invariant by translation:
$$
d(x, y)=\sum_{n=1}^{\infty}\frac{1}{2^{n}}\frac{\|x-y\|_{n}}{1+\|x-y\|_{n}}.
$$

The operator $B: \lambda_{p}(A)\longrightarrow \lambda_{p}(A)$ is continuous and well-defined if
and only if the following condition on the matrix $A$ is satisfied:
\begin{equation}\label{2.5}
\forall n\in \mathbb{N}, \ \exists m>n \text{ such that } \sup_{j\in
\mathbb{N}}\left|\frac{a_{j, n}}{a_{j+1, m}}\right|<+\infty,
\end{equation}
where in the case of $a_{j+1, m}=0$, one has $a_{j, n}=0$ and we consider $\frac{0}{0}$ as 1 (see \cite{Kothe}).

For simplicity, throughout this paper, for any $x=(x_{1}, x_{2}, x_{3}, \ldots)\in \lambda_{p}(A)$ and any $k, n\in \mathbb{Z}^{+}$,
denote
\[
(x)_{k}:=x_{k}, \ x(k):=B^{k}(x)=(x_{k+1}, x_{k+2}, x_{k+3},
\ldots),
\]
\[
x(k, n):=(x_{k+1}, x_{k+2}, \ldots, x_{k+n}, 0, 0, \ldots),
\]
and
\[
x[k, n]:=(\underbrace{0, 0, \ldots, 0}_{n}, x_{k+n+1}, x_{k+n+2},
\ldots).
\]

The notion of distributional chaos was introduced by Schweizer and Sm\'{\i}tal \cite{Schweizer94}. Let $f: X\longrightarrow X$
be a continuous map defined on a metric space $(X, d)$. For any $x, y\in X$, $n\in \mathbb{N}$ and $t\in \mathbb{R}$, let
$$
\Phi_{x, y}^{(n)}(t)=\left|\{0\leq i<n: d(f^{i}(x), f^{i}(y))<t\}\right|.
$$
Define {\it lower} and {\it upper distributional functions}, $\mathbb{R}\longrightarrow [0, 1]$ generated by $f$, $x$ and $y$, as follows:
$$
\Phi_{x, y}(t)=\liminf_{n\rightarrow \infty}\frac{1}{n}\Phi_{x, y}^{(n)}(t),
$$
and
$$
\Phi_{x, y}^{*}(t)=\limsup_{n\rightarrow \infty}\frac{1}{n}\Phi_{x, y}^{(n)}(t),
$$
respectively, where where $|A|$ denotes the cardinality of set $A$. A subset $D\subset X$
is {\it distributionally $\varepsilon$-scrambled} if for any distinct points $x, y\in D$,
$\Phi_{x, y}^{*}(t)=1$ for any $t>0$ and $\Phi_{x, y}(\varepsilon)=0$. A pair satisfying
the above condition is called a {\it distributionally $\varepsilon$-chaotic pair}.

 During the last decades, many research works were devoted to the `chaotic behavior' of the backward shift operator on the K\"{o}the
sequence space (more generally, Banach or Fr\'{e}chet space) (see, e.g., \cite{Bermudez11,BBMP,GE,Martinez-Gimenez02,Martinez-Gimenez07,Martinez-Gimenez09,
Oprocha06,Wu11,Wu121,Wu13,Wu}). For example, Mart\'{\i}nez-Gim\'{e}nez and Peris \cite{Martinez-Gimenez02} obtained some characterizations
for hypercyclicity and Devaney chaos under backward shift on the K\"{o}the sequence space. Mart\'{\i}nez-Gim\'{e}nez \cite{Martinez-Gimenez07}
provided some sufficient conditions for the operator $f(B_{\mathrm{w}})$ to be chaotic in the sense of Devaney. Then, Berm\'{u}dez et al.
\cite{Bermudez11} proved some useful equivalent conditions for Li-Yorke chaos and a few sufficient criteria for distributionally chaotic operators.
Characterizations of hypercyclicity and Li-Yorke chaos for weighted shifts on more general sequence spaces were obtained in \cite{GE} and \cite{BBMP}, respectively. By employing methods developed in \cite{Bermudez11}, Wu and Zhu \cite{Wu121} proved that for a bounded operator defined on a Banach space, Li-Yorke chaos, Li-Yorke sensitivity, spatiotemporal chaos, and distributional chaos in a sequence are all equivalent, and they are all strictly stronger than sensitivity. Further results of \cite{Oprocha06} were extended to maximal distributional chaos for the annihilation operator of a quantum harmonic oscillator in \cite{Wu11,Wu131}. In 2009, Mart\'{\i}nez-Gim\'{e}nez et al. \cite{Martinez-Gimenez09} provided sufficient conditions for uniform distributional chaos under backward shift. Very recently, we \cite{Wu2016,Wu13,Wu,WWC17} provided a class of characterizations for uniform Li-Yorke chaos and a sufficient condition for maximal distributional chaos under backward shift on the K\"{o}the sequence space. Bernardes et al. \cite{Bernardes13} characterized distributional chaos for linear operators on Fr\'{e}chet spaces and obtained a sufficient condition to ensure the existence of dense uniformly distributionally irregular manifolds.

For quite a long time, operator theorists have been studying the so-called cyclic vectors in connection with the (invariant) subspace
problem \cite{Bes99,Golinski12,Godefroy91,Read84,Read88}. The invariant subspace problem, which is open to this day, asks whether every
Hilbert space operator possesses an invariant closed subspace other than the two trivial ones given by $\{0\}$ and the whole space.
Du \cite{Du05} proved that an interval map is turbulent if and only if there is an invariant scrambled set in 2005. Later, Oprocha \cite{Oprocha09}
extended this approach and proved that exactly the same characterization is valid for distributional chaos. Very recently, for the full shift
$(\Sigma_{2}, \sigma)$ on two symbols, we \cite{Wu1} constructed an invariant distributionally $\varepsilon$-scrambled set for any
$0<\varepsilon<\dim(\Sigma_{2})$, in which each point is transitive but is not weakly almost periodic.

\medskip

In \cite{Martinez-Gimenez09}, Mart\'{\i}nez-Gim\'{e}nez et al. proved the following:
\begin{theorem}\label{Th 2}\cite[Theorem 5]{Martinez-Gimenez09}
Let $A$ be a K\"{o}the matrix satisfying (\ref{2.5}), $1\leq
p<+\infty$ (or, $p=0$). If there exist $x, y\in \lambda_{p}(A)$ such
that $\Phi_{x, y}(\delta)=0$ holds for some $\delta>0$, then $B:
\lambda_{p}(A)\longrightarrow\lambda_{p}(A)$ has a
distributionally $\varepsilon$-scrambled subset for some $\varepsilon>0$.
\end{theorem}

Combining this with \cite[Theorem 3.3]{Wu2016}, Wu et al. \cite{Wu2016} provided the following question:
\begin{question}\label{Q1}\cite[Question 3.5]{Wu2016}
Does the hypothesis in Theorem \ref{Th 2} imply that $B$ has an invariant distributionally
scrambled linear manifold?
\end{question}

Being a partial answer to Question \ref{Q1}, this paper shall prove that the hypothesis in Theorem
\ref{Th 2} can ensure that $B$ admits an invariant distributionally $\varepsilon$-scrambled subset
for any $0<\varepsilon<\delta$ (see Theorem \ref{Th 1}).

\begin{theorem}\label{Th 1}
Let $A$ be a K\"{o}the matrix satisfying (\ref{2.5}), $1\leq p<+\infty$ (or, $p=0$). If there exist $x, y\in \lambda_{p}(A)$ such
that $\Phi_{x, y}(\delta)=0$ holds for some $\delta>0$, then $B: \lambda_{p}(A)\longrightarrow\lambda_{p}(A)$ has an
invariant distributionally $\varepsilon$-scrambled subset for any $0<\varepsilon<\delta$.
\end{theorem}

\begin{proof} By the invariability of the metric $d$, we may assume that $x=0$ and $y=(y_{1}, y_{2}, y_{3}, \ldots)$.
Since $\Phi_{x, y}(\delta)=0$, there exists an increasing sequence $\left\{\mathrm{N}_{k}\right\}_{k\in \mathbb{N}}\subset \mathbb{N}$
such that
\begin{equation}\label{4.6}
\lim_{k\rightarrow \infty}\frac{1}{\mathrm{N}_{k}}\left|\{0\leq i<\mathrm{N}_{k}:
d(B^i(x), B^i(y))\geq \delta\}\right|=1.
\end{equation}
It is not difficult to
check that there exists a subsequence $\{\mathrm{M}_{k}\}_{k\in \mathbb{N}}$
of $\left\{\mathrm{N}_{k}\right\}_{k\in \mathbb{N}}$ such that for any $k\in \mathbb{N}$,
\begin{equation}\label{4.7}
\mathrm{M}_{k+1}-\mathrm{M}_{k}\geq 4^{\mathrm{M}_{k}},
\end{equation}
\begin{equation}\label{4.8}
\sum_{j=\mathrm{M}_{k}}^{\infty}\left|y_{j}a_{j,
k}\right|^{p}<\frac{1}{2^{k}},
\end{equation}
and that for any $\mathrm{M}_{2k}< j\leq \mathrm{M}_{2k+1}$,
\begin{equation}\label{4.9}
d\left(0, y\left(j, \mathrm{M}_{2k+2}-\mathrm{M}_{2k+1}\right)\right)\geq d(0,B^{j}(y))-\frac{1}{2k}.
\end{equation}
Define $\nu=(\nu_{1}, \nu_{2}, \nu_{3}, \ldots)$ by
$$
\nu_{j}=
\begin{cases}
k\cdot |y_{j}|, &  \mathrm{M}_{4k}< j\leq \mathrm{M}_{4k+3}, k\in \mathbb{N}, \\
0, & \text{otherwise}.
\end{cases}
$$
Because
\begin{eqnarray*}
\sum_{j=\mathrm{M}_{k}}^{\infty}|\nu_{j}a_{j, k}|^{p}&=&\sum_{l\geq
k}\sum_{j=\mathrm{M}_{l}+1}^{\mathrm{M}_{l+1}}|\nu_{j}a_{j, k}|^{p}\leq \sum_{l\geq
k}\sum_{j=\mathrm{M}_{l}+1}^{\mathrm{M}_{l+1}}l^{p}|y_{j}a_{j, k}|^{p}\\
&\leq& \sum_{l\geq k}\sum_{j=\mathrm{M}_{l}+1}^{\mathrm{M}_{l+1}}l^{p}|y_{j}a_{j,l}|^{p}
\leq \sum_{l\geq k}\frac{l^{p}}{2^{l}}<+\infty,
\end{eqnarray*}
one has $\nu\in \lambda_{p}(A)$. Applying the method of induction, it
can be verified that there exists a subsequence
$\left\{\mathrm{M}_{k_{n}}\right\}_{n\in \mathbb{N}}$ of
$\left\{\mathrm{M}_{k}\right\}_{k\in \mathbb{N}}$ such that for any $n\in
\mathbb{N}$,
$$
\left\{k_{2n}: n\in \mathbb{N}\right\}\subset \left\{4k: k\in
\mathbb{N}\right\}, \quad k_{1}=1, \quad k_{2n+1}=k_{2n}+3,
$$
and that for any $\mathrm{M}_{k_{2n-1}}< j\leq \mathrm{M}_{k_{2n-1}}+2^{M_{k_{2n-1}}}$,
\begin{equation}\label{4.10}
d\left(0, \nu\left[j,\mathrm{M}_{k_{2n}}-4^{\mathrm{M}_{k_{2n-1}}}\right]\right)\leq\frac{1}{2n}.
\end{equation}

Arrange all odd prime numbers by the natural order `$<$' and denote
them as $\mathrm{P}_{1}, \mathrm{P}_{2}, \ldots$. For any $n, l\in
\mathbb{N}$, set  
\begin{eqnarray*}
\mathscr{C}_{n, l}^{0}=\Bigg\{j\in \mathbb{N}:&
\mathrm{M}_{k_{\mathrm{P}_{n}^{l}+1}}+(2u)l<j\leq
\mathrm{M}_{k_{\mathrm{P}_{n}^{l}+1}}+(2u+1)l, \\
& 0\leq 2u\leq
\left[\frac{\mathrm{M}_{k_{\mathrm{P}_{n}^{l}+1}+3}-\mathrm{M}_{k_{\mathrm{p}_{n}^{l}+1}}}{l}\right]-1\Bigg\},
\end{eqnarray*}
and
$$
\mathscr{C}_{n, l}^{1}=\left\{j\in \mathbb{N}:
\mathrm{M}_{k_{\mathrm{P}_{n}^{l}+1}}< j\leq
\mathrm{M}_{k_{\mathrm{P}_{n}^{l}+1}+3}\right\}-\mathscr{C}_{n, l}^{0}.
$$
Take $\bar{\nu}=(\bar{\nu}_{1}, \bar{\nu}_{2}, \bar{\nu}_{3},
\ldots)\in \lambda_{p}(A)$ with
$$
\bar{\nu}_{j}=
\begin{cases}
|\nu_j|, &  j\in \mathscr{C}_{n, l}^{0}, n, l \in \mathbb{N}, \\
-|\nu_j|, &  j\in \mathscr{C}_{n, l}^{1}, n, l\in \mathbb{N}, \\
0, & \text{otherwise},
\end{cases}
$$
and set
$$
\mathscr{D}=\bigcup_{n=0}^{\infty}B^{n}\left(\left\{\alpha \bar{\nu}:
\alpha\in (0, 1)\right\}\right).
$$
Clearly, $B(\mathscr{D})\subset \mathscr{D}$ and $\mathscr{D}$ is uncountable. Given any two fixed points
$a$, $b \in \mathscr{D}$ with $a\neq b$, there exist $\alpha, \beta\in (0, 1)$ and $p, q\in \mathbb{Z}^{+}$ such that
$a=B^{p}(\alpha\bar{\nu})$ and $b=B^{q}(\beta\bar{\nu})$. Without loss of generality, assume that $p\leq q$.

\medskip

Now, we assert that $(a, b)$ is a distributionally $\varepsilon$-chaotic pair for any $0<\varepsilon<\delta$.

\medskip

Firstly, for any $\mathrm{M}_{k_{2n-1}}<j\leq \mathrm{M}_{k_{2n-1}}+2^{M_{k_{2n-1}}}-q$, noting that
$\mathrm{M}_{k_{2n}}-4^{\mathrm{M}_{k_{2n-1}}}\leq \mathrm{M}_{k_{2n}}-(j+q)$, and applying
(\ref{4.10}), it follows that
\begin{eqnarray*}
&&d(B^{j}(a), B^{j}(b))\\
&\leq& d(0, B^{j}(a))+d(0,
B^{j}(b))\\
&\leq& d(0, B^{j+p}(\nu))+d(0, B^{j+q}(\nu))\\
&\leq&\frac{1}{n}\longrightarrow 0, \
(n\longrightarrow \infty).
\end{eqnarray*}
Then, given any $t>0$, there exists some $\mathrm{N}\in \mathbb{N}$ such that for any $n\geq \mathrm{N}$ and any
$\mathrm{M}_{k_{2n-1}}<j\leq \mathrm{M}_{k_{2n-1}}+2^{\mathrm{M}_{k_{2n-1}}}-q$,
$$
d(B^{j}(a), B^{j}(b))<t,
$$
implying that
\begin{eqnarray*}
&&\Phi_{a, b}^{*}(t)\\
&=&\limsup_{n\rightarrow\infty}\frac{1}{n}\left|\{0\leq i<n: d(B^i(a), B^i(b))<t\}\right|\\
&\geq&\limsup_{n\rightarrow\infty}\frac{1}{\mathrm{M}_{k_{2n-1}}+2^{\mathrm{M}_{k_{2n-1}}}-q}
\left|\left\{0\leq i<\mathrm{M}_{k_{2n-1}}+2^{\mathrm{M}_{k_{2n-1}}}-q: d(B^i(a), B^i(b))<t\right\}\right|\\
&\geq&\limsup_{n\rightarrow\infty}\frac{2^{\mathrm{M}_{k_{2n-1}}}-q}{\mathrm{M}_{k_{2n-1}}+2^{\mathrm{M}_{k_{2n-1}}}-q}\\
&=&1.
\end{eqnarray*}

\medskip

Second, to prove $\Phi_{a, b}(\varepsilon)=0$ for any $0<\varepsilon<\delta$, we consider two cases as follows:

\medskip

\textbf{Case 1.} $p=q$ and $\alpha\neq\beta$. Noting that for any
$\mathrm{M}_{k_{\mathrm{P}_{n}+1}}< j\leq \mathrm{M}_{k_{\mathrm{P}_{n}+1}+1}$ and any $1\leq i\leq
\mathrm{M}_{k_{\mathrm{P}_{n}+1}+2}-\mathrm{M}_{k_{\mathrm{P}_{n}+1}+1}$,
$\left|\left(B^{j}\left((\alpha-\beta)
\bar{\nu}\right)\right)_{i}\right|= \big|\left((\alpha-\beta)
\frac{k_{\mathrm{P}_{n}+1}}{4}y(j)\right)_{i}\big|$, and applying (\ref{4.9}),
it follows that for all $n\in\mathbb{N}$ with
$\left|\alpha-\beta\right| \frac{k_{\mathrm{P}_{n}+1}}{4}>1$ and any
$\mathrm{M}_{k_{\mathrm{P}_{n}+1}}< j\leq \mathrm{M}_{k_{\mathrm{P}_{n}+1}+1}$,
\begin{eqnarray*}
&&d(B^{j}(\alpha\bar{\nu}), B^{j}(\beta\bar{\nu}))\\
&=&d(0, B^{j}((\alpha-\beta)\bar{\nu}))\\
&\geq & d\left(0, y\left(j, \mathrm{M}_{k_{\mathrm{P}_{n}+1}+2}-\mathrm{M}_{k_{\mathrm{P}_{n}+1}+1}\right)\right)\\
&\geq& d(0,B^{j}(y))-\frac{1}{k_{\mathrm{P}_{n}+1}}.
\end{eqnarray*}
This, together with (\ref{4.6}) and (\ref{4.7}), implies that for any $0<\varepsilon<\delta$,
\begin{eqnarray*}
&&\Phi_{a, b}(\varepsilon)\\
&=&1-\limsup_{n\rightarrow \infty}\frac{1}{n}\left|\{0\leq i<n: d(B^i(a), B^i(b))\geq \varepsilon\}\right|\\
&\leq& 1-\limsup_{n\rightarrow\infty}\frac{1}{\mathrm{M}_{k_{\mathrm{P}_{n}+1}+1}}\left|\{0\leq i<\mathrm{M}_{k_{\mathrm{P}_{n}+1}+1}: d(B^i(a), B^i(b))\geq \varepsilon\}\right|\\
&\leq& 1-\limsup_{n\rightarrow \infty}
\frac{\Phi_{0, y}^{(\mathrm{M}_{k_{\mathrm{P}_{n}+1}+1})}(\delta)-\mathrm{M}_{k_{\mathrm{P}_{n}+1}}}{\mathrm{M}_{k_{\mathrm{P}_{n}+1}+1}}\\
&=&0.
\end{eqnarray*}

\textbf{Case 2.} $q>p$. Fix any $n\in \mathbb{N}$ with
$\alpha\frac{k_{\mathrm{P}_{n}^{q-p}+1}}{4}>1$.
For any $\mathrm{M}_{k_{\mathrm{P}_{n}^{q-p}+1}}<i \leq
\mathrm{M}_{k_{\mathrm{P}_{n}^{q-p}+1}+1}$ and any $i+1\leq j <i+(\mathrm{M}_{k_{\mathrm{P}_{n}^{q-p}}+3}-
\mathrm{M}_{k_{\mathrm{P}_{n}^{q-p}}+2})$, noting that $\left|(\alpha\bar{\nu}-B^{q-p}(\beta\bar{\nu}))_j\right|\geq
\left|(\alpha\bar{\nu})_j\right|$ (as $(\alpha\bar{\nu})_j$ and $(B^{q-p}(\beta\bar{\nu}))_j$
are of different signs), it follows that
\begin{eqnarray*}
&&d\left(B^{i-p}(a), B^{i-p}(b)\right)\\
&=&d(0, B^{i}(\alpha\bar{\nu}-B^{q-p}(\beta\bar{\nu})))\\
&\geq& d(0, \alpha\bar{\nu}(i, \mathrm{M}_{k_{\mathrm{P}_{n}^{q-p}}+3}-
\mathrm{M}_{k_{\mathrm{P}_{n}^{q-p}}+2})).
\end{eqnarray*}
Similarly to the proof of Case 1, it can be verified that for any $0<\varepsilon<\delta$,
$\Phi_{a, b}(\varepsilon)=0$.

Therefore, $\mathscr{D}$ is an invariant distributionally $\varepsilon$-scrambled subset for any $0<\varepsilon<\delta$.
\end{proof}

\begin{remark}
\begin{enumerate}[(1)]
\item Given a sequence $\{\mathrm{w}_{i}\}_{i\geq 2}$ of strictly positive scalars, consider its associated {\it weighted backward shift}
$$
B_{\mathrm{w}}(x_1, x_2, \ldots):=(\mathrm{w}_{2}x_2,
\mathrm{w}_3x_3, \ldots).
$$
According to the discussions in \cite{Martinez-Gimenez07, Martinez-Gimenez09}, the study of chaos under a weighted
backward shift can be reduced to the unweighted case, with a suitable K\"{o}the matrix. So, for weighted backward shift, we
actually have also obtained similar result. \item Applying Theorem \ref{Th 1}, it is easy to verify that all examples in
\cite{Martinez-Gimenez09} admit an invariant $\varepsilon$-scrambled subset for some $\varepsilon>0$.
\item Combining Theorem \ref{Th 1} with \cite[Theorem 2.1]{Wu}, it follows that \cite[Theorem 3.1]{WWC17} holds trivially.
\end{enumerate}
\end{remark}

\section*{Acknowledgments}
{\footnotesize This work was supported by the National Natural Science Foundation of China (No. 11601449),
the National Nature Science Foundation of China (Key Program) (No. 51534006), Science and Technology Innovation Team of Education
Department of Sichuan for Dynamical System and its Applications (No. 18TD0013), Youth Science and Technology Innovation Team of
Southwest Petroleum University for Nonlinear Systems (No. 2017CXTD02), and Scientific Research Starting Project of Southwest Petroleum
University (No. 2015QHZ029).}


\baselineskip=2pt
\begin{thebibliography}{20}

%

\bibitem{Bermudez11}
{\sc T. Berm\'{u}dez, A. Bonilla, F.
Mart\'{i}nez-Gim\'{e}nez and A. Peris},
{\it Li-Yorke and distributionally
chaotic operators}, J. Math. Anal. Appl., {\bf 373} (2011), 83--93.

\bibitem{Bernardes13}
{\sc N. C. Bernardes Jr., A. Bonilla, V. M\"{u}ller and A. Peris},
{\it Distributional chaos for linear operators}, J. Funct. Anal., {\bf 265} (2013), 2143--2163.

\bibitem{BBMP}
{\sc N. C. Bernardes Jr., A. Bonilla, V. M\"{u}ller and A. Peris},
{\it Li-Yorke chaos in linear dynamics},
Ergodic Theory Dynam. Systems, {\bf 35} (2015), 1723--1745.

\bibitem{Bes99}
{\sc J. B\`{e}s},
{\it Invariant manifolds of hypercyclic
vectors for thr real scalar case}, Proc. Amer. Math. Soc.,
{\bf 127} (1999), 1801--1804.

\bibitem{Du05}
{\sc B. Du},
{\it On the invariance of Li-Yorke chaos of interval
maps}, J. Difference Equ. Appl., {\bf 11} (2005),
823--828.

\bibitem{Golinski12}
{\sc M. Goli\'{n}ski},
{\it Invariant subspace problem for classical spaces
of functions}, J. Funct. Anal., {\bf 262} (2012), 1251--1273.

\bibitem{Godefroy91}
{\sc G. Godefroy and J. H. Shapiro}, {\it Operators with
dense, invariant, cyclic vector manifolds}, J. Funct. Anal.,
{\bf 98} (1991), 229--269.

\bibitem{GE}
{\sc K.-G. Grosse-Erdmann},
{\it Hypercyclic and chaotic weighted shifts},
Studia Math., {\bf 139} (2000), 47--68.

\bibitem{Kothe}
{\sc G. K\"{o}the},
{\it Topological Vector Space. I}, translated from German by
D. J. H. Garling, Die Grundlehren der mathematischen Wissenschaften, Band 159, Springer-Verlag New
York Inc., New York, 1969.

\bibitem{Martinez-Gimenez02}
{\sc F. Mart\'{\i}nez-Gim\'{e}nez and A. Peris},
{\it Chaos for backward shift operators}, Internat. J. Bifur. Chaos Appl. Sci.
Engrg., {\bf 12} (2002), 1703--1715.

\bibitem{Martinez-Gimenez07}
{\sc F. Mart\'{\i}nez-Gim\'{e}nez},
{\it Chaos for power series of backward shift operators}, Proc. Amer. Math. Soc.,
{\bf 135} (2007), 1741--1752.

\bibitem{Martinez-Gimenez09}
{\sc F. Mart\'{\i}nez-Gim\'{e}nez, P. Oprocha and A. Peris},
{\it Distributional chaos for backward shift}, J. Math. Anal.
Appl., {\bf 351} (2009), 607--615.

\bibitem{Oprocha06}
{\sc P. Oprocha}, {\it A quantum harmonic oscillator and strong chaos},
J. Phys. A: Math. Gen., {\bf 39} (2006), 14559--14565.

\bibitem{Oprocha09}
{\sc P. Oprocha},
{\it Invariant scrambled sets and distributional chaos},
Dynamical Systems, {\bf 24} (2009), 31--43.

\bibitem{Read84}
{\sc C. J. Read},
{\it A solution to the invariant subspace
problem}, Bull. London Math. Soc., {\bf 16} (1984),
337--401.

\bibitem{Read88}
{\sc C. J. Read}, {\it The invariant subspace problem for a
class of Banach spaces. II. Hypercyclic operators}, Israel J.
Math., {\bf 63} (1988), 1--40.

\bibitem{Schweizer94}
{\sc B. Schweizer and J. Sm\'{\i}tal}, {\it Measures of chaos and a spectral
decomposition of dynamical systems on the interval}, Trans.
Amer. Math. Soc., {\bf 344} (1994), 737--754.

\bibitem{Wu2016}
{\sc X. Wu, G. Chen and P. Zhu},
{\it Invariance of chaos from backward shift on the K\"{o}the sequence
space}, Nonlinearity, {\bf 27} (2014), 271--288.

\bibitem{Wu11}
{\sc X. Wu and P. Zhu},
{\it The principal measure of a quantum harmonic oscillator},
J. Phys. A: Math. Theor., {\bf 44} (2011), 505101 (6pp).

\bibitem{Wu121}
{\sc X. Wu and P. Zhu}, {\it On the equivalence of four chaotic operators},
Appl. Math. Lett., {\bf 25} (2012) 545--549.

\bibitem{Wu13}
{\sc X. Wu and P. Zhu}, {\it Li-Yorke chaos for backward shift operators on K\"{o}the
sequence spaces}, Topology Appl., {\bf 160} (2013), 924--929.

\bibitem{Wu}
{\sc X. Wu}, {\it Maximal distributional chaos of backward shift operators on
K\"{o}the sequence spaces}, Czechoslovak Math. J., {\bf 64} (2014), 105--114.

\bibitem{Wu1}
{\sc X. Wu and P. Zhu}, {\it Invariant scrambled set and maximal
   distributional chaos}, Annales Polonici Mathematici, {\bf 109} (2013), 271--278.

\bibitem{Wu131}
{\sc X. Wu and G. Chen}, {\it On the invariance of maximal
distributional chaos under an annihilation operator}, Appl. Math.
Lett., {\bf 26} (2013), 1134--1140.

\bibitem{WWC17}
{\sc X. Wu, L. Wang and G. Chen}, {\it Weighted backward
shift operators with invariant distributionally
scrambled subsets}, Ann. Funct. Anal., {\bf 8} (2017), 199--210.

\end{thebibliography}
\end{document}